\documentclass[reqno]{amsart}
\usepackage{hyperref}
\newtheorem*{theoA}{Theorem A}
\newtheorem*{theoB}{Theorem B}
\newtheorem*{theoC}{Theorem C}
\newtheorem*{theoD}{Theorem D}
\newtheorem*{theoE}{Theorem E}
\newtheorem*{theoF}{Theorem F}
\newtheorem*{theoG}{Theorem G}
\newtheorem*{theoH}{Theorem H}
\newtheorem*{theoI}{Theorem I}
\newtheorem*{theoJ}{Theorem J}
\newtheorem{theo}{Theorem}[section]
\newtheorem{lem}{Lemma}[section]

\newtheorem{ques}{Question}[section]

\newtheorem{defi}{Definition}[section]
\newtheorem{rem}{Remark}[section]
\newcommand{\ol}{\overline}
\newcommand{\be}{\begin{equation}}
\newcommand{\ee}{\end{equation}}
\newcommand{\beas}{\begin{eqnarray*}}
	\newcommand{\eeas}{\end{eqnarray*}}
\newcommand{\bea}{\begin{eqnarray}}
\newcommand{\eea}{\end{eqnarray}}

\numberwithin{equation}{section}

\begin{document}
\title[\hfilneg \hfil Some further q-shift difference results on Hayman conjecture]
{Some further q-shift difference results on Hayman conjecture}
\author[G. Haldar\hfil
\hfilneg]
{Goutam Haldar}


\address{G. Haldar  \newline
Department of Mathematics, Malda College, West Bengal 732101, India.}
\email{goutamiit1986@gmail.com, goutamiitm@gmail.com}


\maketitle

\begin{abstract}
   In this paper, we investigate the zero distributions of $q$-shift difference-differential polynomials of meromorphic functions with zero-order that extends and generalizes the classical Hayman results of the zeros of differential polynomials to q-shift difference. We also investigate the uniqueness problem of $q$-shift difference-differential polynomials sharing a polynomial value with finite weight. 
   \end{abstract}
\section{\textbf{Introduction}}
A meromorphic function $f(z)$ in the complex plane $\mathbb{C}$ means is a function that is analytic in $\mathbb{C}$ except for the set of isolated points, which are poles of the function. If no poles
occur, then $f(z)$ is called an entire function. Let $f$ and $g$ be two non-constant meromorphic functions defined in the open complex plane $\mathbb{C}$. If for some $a\in\mathbb{C}\cup\{\infty\}$, the zero of  $f-a$ and $g-a$ have the same locations as well as multiplicities, we say that $f$ and $g$ share the value $a$ CM (counting multiplicities). If we do not consider the multiplicities, then $f$ and $g$ are said $a$ IM (ignoring multiplicities). We adopt the standard notations of the Nevanlinna theory of meromorphic functions explained in (\cite{Hayman & 1964}).\par A meromorphic function $\alpha(z)$ is said to be small with respect to $f$ if $T(r,\alpha)=S(r,f)$. i.e., $T(r,\alpha)=o(T(r,f))$ as $r\longrightarrow \infty$, outside of a possible exceptional set of finite linear measure.\par
For a set $S\subset\mathbb{C}$, we define
\begin{equation*} E_{f}(S)=\bigcup_{a\in S}\{z|f(z)=a(z)\},\end{equation*} where each zero is counted according to its multiplicity and $$\overline E_{f}(S)=\bigcup_{a\in S}\{z|f(z)=a(z)
\},\;\text{where each zero is counted only once}.$$\par
If $E_{f}(S)=E_{g}(S)$, we say that $f$, $g$ share the set S CM and if $\overline E_{f}(S)=\overline E_{g}(S)$, we say $f$, $g$ share the set S IM.\par
In 2001, Lahiri (\cite{Lahiri & Nagoya & 2001}) introduced a gradation of sharing of values or sets which is known as weighted sharing. Below we are recalling the notion.
\begin{defi}(\cite{Lahiri & Nagoya & 2001})
	Let $k$ be a non-negative integer or infinity. For $a\in \mathbb{C}\cup\{\infty\}$ we denote by $E_{k}(a,f)$ the set of all $a$-points of $f$, where an $a$ point of multiplicity $m$ is counted $m$ times if $m\leq k$ and $k+1$ times if $m>k.$ If $E_{k}(a,f)=E_{k}(a,g),$ we say that $f$, $g$ share the value $a$ with weight $k$. \end{defi}
We write $f$, $g$ share $(a,k)$ to mean that $f,$ $g$ share the value $a$ with weight $k.$ Clearly if $f,$ $g$ share $(a,k)$ then $f,$ $g$ share $(a,p)$ for any integer $p$, $0\leq p<k.$ Also we note that $f,$ $g$ share a value $a$ IM or CM if and only if $f,$ $g$ share $(a,0)$ or $(a,\infty)$ respectively.\par
\begin{defi} \cite {Alzahare & Yi & 2004} Let $f$ and $g$ be two non-constant meromorphic functions such that $f$ and $g$ share the value $a$ IM. Let $z_{0}$ be a $a$-point of $f$ with multiplicity $p$, a $a$-point of $g$ with multiplicity $q$. We denote by $\ol N_{L}(r,a;f)$ the counting function of those $a$-points of $f$ and $g$ where $p>q$, by $N^{1)}_{E}(r,a;f)$ the counting function of those $a$-points of $f$ and $g$ where $p=q=1$ and by $\ol N^{(2}_{E}(r,a;f)$ the counting function of those $a$-points of $f$ and $g$ where $p=q\geq 2$, each point in these counting functions is counted only once. Similarly, one can define $\ol N_{L}(r,a;g),\; N^{1)}_{E}(r,a;g),\; \ol N^{(2}_{E}(r,a;g).$
\end{defi}
\begin{defi}\cite{Lahiri & Nagoya & 2001,Lahiri & Complex Var & 2001} Let $f$, $g$ share a value $a$ IM. We denote by $\ol N_{*}(r,a;f,g)$ the reduced counting function of those $a$-points of $f$ whose multiplicities differ from the multiplicities of the corresponding $a$-points of $g$.
	Clearly $\ol N_{*}(r,a;f,g)\equiv\ol N_{*}(r,a;g,f)$ and $\ol N_{*}(r,a;f,g)=\ol N_{L}(r,a;f)+\ol N_{L}(r,a;g)$.
\end{defi}
\begin{defi}\cite{Lahiri & Sarkar & 2004}Let $p$ be a positive integer and $a\in\mathbb{C}\cup\{\infty\}$.\begin{enumerate}
		\item[(i)] $N(r,a;f\mid \geq p)$ ($\ol N(r,a;f\mid \geq p)$)denotes the counting function (reduced counting function) of those $a$-points of $f$ whose multiplicities are not less than $p$.\item[(ii)]$N(r,a;f\mid \leq p)$ ($\ol N(r,a;f\mid \leq p)$)denotes the counting function (reduced counting function) of those $a$-points of $f$ whose multiplicities are not greater than $p$.
	\end{enumerate}
\end{defi}
In recent times, many mathematicians are working on difference equations, the difference product and the q-difference analogues the value distribution theory of entire and meromorphic functions in the complex plane (see \cite{Chiang & Feng & 2008}, \cite{Halburd & Korhonen & 2006}, \cite{Halburd & Korhonen & 2006 & Fenn}, \cite{Heittokangas et al & 2011}, \cite{Laine & Yang & 2007}, \cite{Liu & 2009}, \cite{Liu & Yang & 2009}, \cite{Zhang & 2005}). In 2006, Halburd and Korhonen (\cite{Halburd & Korhonen & 2006}) established a difference analogue of the Logarithmic Derivative Lemma, and then applying it, a lot of results on meromorphic solutions of complex difference equations has been proved. After that Barnett, Halburd, Korhonen and Morgan (\cite{Barnett & Halburd & 2007}) also established a q-difference analogue of the Logarithmic Derivative Lemma. \par 
Let us first recall the notion of the q-shift and q-difference operator of a meromorphic function $f$.
\begin{defi}
	For a meromorphic function f and $c,\; q(\neq 0)\in\mathbb{C}$, let us now denote its q-shift $E_{q}f$ and q-difference operators $\Delta_{q}f$ respectively by $E_{q}f(z) = f(qz+c)$ and $\Delta_{q}f(z)=f(qz+c)-f(z)$.\end{defi}
For further generalization of $\Delta_{q}f(z)$, we now define the $q$-difference operator of an entire (meromorphic) function $f$ as as $L(z,f)=b_1f(qz+c)+b_0f(z)$, where $b_1(\neq0)$ and $b_0$ are complex constants.
For $s\in \mathbb{N}$, let us define \beas \chi_{_{b_{0}}}=\begin{cases} 1, \;\text{if}\; b_0\neq0\\ 0, \; \text{if}\; b_0=0.\end{cases}\eeas
\par Let $P(z)=a_mz^m+a_{m-1}z^{m-1}+\ldots+a_0$ be a nonzero polynomial of degree $n$, where $a_{m}(\neq 0), a_{m-1}, \ldots, a_0$ are complex constants and $m$ is a positive integer. Let $m_1$ be the number of distinct simple zeros and $m_2$ be the number of distinct multiple zeros of $P(z)$. Let $\Gamma_{0}=m_1+2m_2$ and $\Gamma_{1}=m_1+m_2$.
\par 
In $2010$, Zhang and Korhonen  (\cite{Zhang & Korhonen & 2010}) studied the value distribution of $q$-difference polynomials of meromorphic functions and obtained the following result.
\begin{theoA}\cite{Zhang & Korhonen & 2010}
	Let $f$ be a transcendental meromorphic (resp. entire) function of zero order and q non-zero complex constant. Then for $n\geq6$	(resp. $n \geq2$), $f(z)^nf(qz)$ assumes every non-zero value $a\in\mathbb{C}$ infinitely often.\end{theoA}
Recently, Liu and Qi \cite{Liu & Qi & 2011} firstly investigated value distributions for a q-shift of the meromorphic function and obtained the following result.
\begin{theoB}\cite{Liu & Qi & 2011}
	Let $f$ be a zero-order transcendental meromorphic function, $n\geq6$; $q\in\mathbb{C}-\{0\}$, $\eta\in\mathbb{C}$ and $R(z)$ a rational function. Then $f(z)^nf(qz + \eta)-R(z)$ has infinitely many zeros.\end{theoB}
In 2011, Liu, Liu and Cao \cite{Liu & Liu & Cao & 2011} investigated about the zeros of $f^n(f^m-1)[f(qz+c)-f(z)]=a(z)$, where $a(z)$ is a small function of $f$ and obtained the following result.
\begin{theoC}\cite{Liu & Liu & Cao & 2011}
	Let $f$ be a transcendental meromorphic (resp. entire) function with zero-order. If $n\geq7$ (resp. $n\geq3$), then $f(z)^n( f(z)^m-a)[f(qz+c)-f(z)]-\alpha(z)$ has
	infinitely many zeros, where $\alpha(z)$ is a non-zero small function with respect to $f$.\end{theoC}
Later on Xu, Liu and Cao \cite{Xu & Liu & Cao & 2015} started investigation about the zeros of $P(f)f(qz + \eta) = a(z)$and $P(f)[f(qz + \eta)-f(z)] = a(z)$ and obtained the following two results. \par 
\begin{theoD}\cite{Xu & Liu & Cao & 2015}
	Let $f$ be a zero-order transcendental meromorphic (resp. entire) function, $q(\neq 0), \eta$ are complex constants. Then for $n > m+4$ (resp. $n > m$), $P(f)f(qz+\eta)=a(z)$ has infinitely many solutions, where a(z) is a non-zero small functions in $f$.\end{theoD}
\begin{theoE}\cite{Xu & Liu & Cao & 2015}
	Let $f$ be a zero-order transcendental meromorphic (resp. entire) function, $q(\neq 0), \eta$ are complex constants. Then for $n > m+6$ (resp. $n > m+2$), $P(f)\{f(qz+\eta)-f(z)\}=a(z)$ has infinitely many solutions, where a(z) is a non-zero small functions in $f$.\end{theoE}
For the uniqueness of difference and $q$-difference of meromorphic functions Zhang and Korhonen \cite{Zhang & Korhonen & 2010} obtained the following results.
\begin{theoF}\cite{Zhang & Korhonen & 2010}
Let $f$ and $g$ be two transcendental meromorphic (resp. entire) functions of zero order. Suppose that $q$ is a non-zero complex constant and $n$ is an integer satisfying $n\geq8$ (resp. $n\geq4$). If $f(z)^nf(qz)$ and $g(z)^ng(qz)$ share $1$, $\infty$ CM, then $f(z)\equiv tg(z)$ for $t^{n+1}=1$.
\end{theoF}
\begin{theoG}\cite{Zhang & Korhonen & 2010}
	Let $f$ and $g$ be two transcendental entire functions of zero order. Suppose that $q$ is a non-zero complex constant and $n$ is an integer satisfying $n\geq6$. If $f(z)^n(f(z)-1)f(qz)$ and $g(z)^n(g(z)-1)g(qz)$ share $1$ CM, then $f(z)\equiv g(z)$.
\end{theoG}
To generalize the above result, Liu, Liu and Cao \cite{Liu & Liu & Cao & 2011}, obtained the following result.
\begin{theoH}\cite{Liu & Liu & Cao & 2011}
	Let $f(z)$ and $g(z)$ be transcendental entire functions with zero order. If $n\geq m +5$, $f(z)^n(f(z)^m-a)f(qz+c)$ and $g(z)^n(g(z)^m-a)g(qz +c)$ share a nonzero polynomial $p(z)$ CM, then $f(z)\equiv tg(z)$, where $t^{n+1}=t^m=1$.
\end{theoH}
In 2015, Xu, Liu and Cao \cite{Xu & Liu & Cao & 2015} also investigated the uniqueness problems of $q$-shift of entire functions and obtained the following results.
\begin{theoI}\cite{Xu & Liu & Cao & 2015}
	Let $f$ and $g$ be transcendental entire functions of finite order, $c$ a nonzero complex constant and let $n>2\Gamma_{0}+1$ be an integer. If $P(f)f(qz+c)$ and $P(g)g(qz+c)$ share $1$ CM, then one of the following results holds:
	\begin{enumerate}
		\item[(i).] $f\equiv tg$ for a constant $t$ such that $t^d=1$, where $d=\text{GCD}\{\lambda_0,\lambda_1,\ldots,\lambda_m\}$ and \beas \lambda_i=\begin{cases} i+1, \;\text{if}\; a_i\neq0\\ n+1, \; \text{if}\; a_0=0.\end{cases}, i=0,1,\ldots,m;\eeas
		\item[(ii).]  $f$ and $g$ satisfy the algebraic equation $R(f; g)=0$, where $R(w_1; w_2)=P(w_1)w_1(z+c)-P(w_2)w_2(z+c)$;
		\item[(iii).]  $fg\equiv\mu$, where $\mu$ is a complex constant satisfying $a_n^2\mu^{n+1}\equiv 1$.
\end{enumerate}\end{theoI}
\begin{theoJ}\cite{Xu & Liu & Cao & 2015}
	Let $f$ and $g$ be two transcendental entire functions of zero order, and let $q\in \mathbb{C}-\{0\}$, $c\in\mathbb{C}$. Suppose $E_l(1; P(f)f(qz+c))=E_l(1; P(g)g(qz+c))$ and $l$, $m$, $n$ are integers satisfying one of the following conditions :
	\begin{enumerate}
		\item[(i).] $l\geq 3$, $n>2\Gamma_{0}+1$;
		
		\item[(ii).]  $l=2$, $n>2\Gamma_{0}+\Gamma_{1}+2-\lambda$;
		\item[(iii).]  $l=1$, $n>2\Gamma_{0}+2\Gamma_{1}+3-2\lambda$;
		\item[(iv).]  $l=0$, $n>2\Gamma_{0}+3\Gamma_{1}+4-3\lambda$. Then the conclusions of Theorem I hold, where $\lambda=\textit{min}\{\Theta(0;f);\Theta(0; g)\}$.
\end{enumerate}\end{theoJ}
\begin{ques}\label{q1.1}
	Can we unify Theorems A--E into a single result that improves, extends, and generalizes all the results to a large extent?
\end{ques}
\begin{ques}\label{q1.2}
	Can we further generalize Theroem H by replacing  $(f(z)^m-a)$ ($(g(z)^m-a)$) in place of $P(f)$ ($P(g)$) and $L(z,f)$ ($L(z,g)$) in place of $f(qz+c)$ ($g(qz+c)$) under relax sharing hypothesis?
\end{ques}
\begin{ques}\label{q1.3}
	What can be said about uniqueness if we replace the difference polynomial $P(f)f(qz+c)$ with more general settings differential-difference polynomial $(P(z)L(z,f))^{k}$ in Theorems I and J?
\end{ques}
The purpose of the paper is to find out all the possible answers to the above questions. Section 2 includes our main results.
\section{\textbf{Main Results}}
\begin{theo}\label{t1}
	Let $f$ be a transcendental meromorphic (resp., entire) function of zero order, and $k$ be a non-negative integer. If $m> \Gamma_{1}+km_2+2s+(1+s)\chi_{_{b_{0}}}+2$ (resp., $n> \Gamma_{1}+km_2$), then $(P(f)L(z,f)^s)^{(k)}-\alpha(z)$ has infinitely many zeros, where $\alpha(z)\in S(f)-\{0\}$.
\end{theo}
\begin{theo}\label{t2}
	Let $f(z)$ and $g(z)$ be two transcendental entire functions of zero order and $n$ be a positive integer such that $n\geq m+5$. Let $f^nP(f)L(z,f)-p(z)$ and $g^nP(g)L(z,g)-p(z)$ share $(0,2)$, where $p(z)$ be a
	nonzero polynomial such that $\text{deg}(p)< (n-1)/2$ and $g(z), g(qz+c)$ share $0$ CM. Then one of the following conclusions can be realized:
	\begin{enumerate}
		\item[(i).] $ f\equiv tg $, where $ t $ is a constant satisfying $ t^d $, where $ d=GCD(n+m+1, n+m, \cdots, n+1) $ and $ a_{q-j}\neq 0 $ for some $ j=0, 1,\cdots, m. $
		\item[(ii).]  $ f $ and $ g $ satisfy the algebraic equation $ \mathcal{A}(x,y)=0 $, where \beas \mathcal{A}(w_1,w_2)&=& w_1^n\left(a_mw_1^m+\cdots+a_1w_1+a_0\right)L(z,w_1)\\&-& w_2^n\left(a_mw_2^m+\cdots+a_1w_2+a_0\right)L(z,w_2). \eeas
		\end{enumerate}
\end{theo}
\begin{theo}\label{t3}
	Let $f$, $g$ be two transcendental entire functions of zero order. If $E_{l}(1; (P(f)L(z,f))^{(k)})=E_{l}(1; (P(g)L(z,g))^{(k)})$ and $l,m,n$ are integers satisfying one of the following conditions:
	\begin{enumerate}
		\item[(i).] $ l\geq 2 $, $m>2\Gamma_{0}+2km_2+1$
		\item[(ii).]  $ l=1 $, $m>\frac{1}{2}(\Gamma_{1}+4\Gamma_{0}+5km_2+3)$
		\item[(iii).]  $ l=0 $, $m>(3\Gamma_{1}+2\Gamma_{0}+5km_2+4)$,
	\end{enumerate}
then one of the following results holds:
\begin{enumerate}
	\item[(i).] $ f\equiv tg $ for a constant $t$ such that $t^{d}=1$, where $d=\text{GCD}(\lambda_0, \lambda_1,\ldots, \lambda_m)$ 
	\item[(ii).]  $f$ and $g$ satisfy the algebraic equation $\mathcal{A}(w_1,w_2)=0$, where \beas \mathcal{A}(w_1,w_2)=P(w_1)L(z,w_1)-P(w_2)L(z,w_2).\eeas
\end{enumerate}\end{theo}
\begin{rem}\label{r1.1a}
	Let $P(z)=z^{n}$, where $n$ is a positive integer. Suppose $f(z)=e^{-z}$, $q=-n$, $c\in\mathbb{C}$ and $\alpha(z)=e^{-c}$. Then $P(f(z))[f(qz+c)-f(z)]-\alpha(z)=-e^{-(n+1)z}$ have no zeros. This shows that zero order growth restriction in Theorem \ref{t1} can not be extended to finite order.
\end{rem}
\begin{rem}\label{r2.2}
		\begin{enumerate}
			\item[(i).] Theorem \ref{t1} is a more general result than Theorems A--E.
			\item[(ii).]  Theorems \ref{t2} and \ref{t3} answer the answer of questions \ref{q1.2} and \ref{q1.3}, respectively.\end{enumerate}
\end{rem}

\section{\textbf{Some Lemmas}}
We now prove several lemmas which will play key roles in proving the main results of the paper. Let $\mathcal{F}$ and $\mathcal{G}$ be two non-constant meromorphic functions. Henceforth we shall denote by $\mathcal{H}$ the following function \be\label{e3.1}\mathcal{H}=\left(\frac{\;\;\mathcal{F}^{\prime\prime'}}{\mathcal{F}^{\prime}}-\frac{2\mathcal{F}^{\prime}}{\mathcal{F}-1}\right)-\left(\frac{\;\;\mathcal{G}^{\prime\prime}}{\mathcal{G}^{\prime}}-\frac{2\mathcal{G}^{\prime}}{\mathcal{G}-1}\right).\ee
\begin{lem}\label{lem3.1}\cite{Liu & Qi & 2011}
	Let $f$ be a zero-order meromorphic function, and $c, q(\neq 0)\in\mathbb{C}$. Then \beas m\left(r,\frac{f(qz+c)}{f(z)}.\right)=S(r,f).\eeas 
	\end{lem}
\begin{lem}\label{lem2.2}\cite{Xu & Liu & Cao & 2015}
	Let $f$ be a zero-order meromorphic function, and $c, q\in\mathbb{C}$. Then
	\beas T(r,f(qz+c))=T(r,f)+S(r,f),\eeas  \beas N(r,\infty;f(qz+c))=N(r,\infty;f(z))+S(r,f),\eeas  \beas N(r,0;f(qz+c))=N(r,0;f(z))+S(r,f)\eeas 
	\beas \ol N(r,\infty;f(qz+c))=\ol N(r,\infty;f(z))+S(r,f),\eeas  \beas \ol N(r,0;f(qz+c))=\ol N(r,0;f(z))+S(r,f)\eeas on a set of logarithmic density $1$.
\end{lem}

\begin{lem}\label{lem3.3}\cite{7a} If $N\left(r,0;f^{(k)}\mid f\not=0\right)$ denotes the counting function of those zeros of  $f^{(k)}$ which are not the zeros of $f$, where a zero of $f^{(k)}$ is counted according to its multiplicity then $$N\left(r,0;f^{(k)}\mid f\not=0\right)\leq k\ol N(r,\infty;f)+N\left(r,0;f\mid <k\right)+k\ol N\left(r,0f\mid\geq k\right)+S(r,f).$$\end{lem}
\begin{lem}\label{lem3.4}\cite{9} Let $f$ be a non-constant meromorphic function and let \[\mathcal{R}(f)=\frac{\sum\limits _{i=0}^{n} a_{i}f^{i}}{\sum \limits_{j=0}^{m} b_{j}f^{j}}\] be an irreducible rational function in $f$ with constant coefficients $\{a_{i}\}$ and $\{b_{j}\}$ where $a_{n}\not=0$ and $b_{m}\not=0$. Then $$T(r,\mathcal{R}(f))=d\;T(r,f)+S(r,f),$$ where $d=\max\{n,m\}$.\end{lem}
\begin{lem}\label{lem3.5}\cite{Lahiri & Complex Var & 2001}
	Let $\mathcal{F}$ and $\mathcal{G}$ be two non-constant meromorphic functions satisfying $E_\mathcal{F}(1,m) = E_\mathcal{G}(1,m)$, $0\leq m <\infty$ with $\mathcal{H}\not\equiv0$, then \beas N_E^{1)}(r,1;\mathcal{F})\leq N(r,\infty;\mathcal{H})+S(r,\mathcal{F})+S(r, \mathcal{G}).\eeas Similar inequality holds for $\mathcal{G}$ also. \end{lem}
\begin{lem}\label{lem2.11}\cite{Yi & 1999}
	Let $\mathcal{H}\equiv0$ and $\mathcal{F}$, $\mathcal{G}$ share $(\infty, 0)$, then $\mathcal{F}$, $\mathcal{G}$ share $(1,\infty)$, $(\infty,\infty)$.\end{lem}
	
\begin{lem}\label{lem3.7}\cite{Lahiri & Banerjee & 2006}
	Suppose $\mathcal{F}$, $\mathcal{G}$ share $(1,0)$, $(\infty,0)$. If $\mathcal{H}\not \equiv 0,$ then, \beas N(r,\infty;\mathcal{H}) &\leq& N(r,0;\mathcal{F} \mid\geq 2) + N(r,0;\mathcal{G} \mid\geq 2)+\ol N_{*}(r,1;\mathcal{F}, \mathcal{G}) + \ol N_{*}(r,\infty;\mathcal{F}, \mathcal{G}) \\&&+ \ol N_{0}(r,0;\mathcal{F}^{\prime}) + \ol N_{0}(r,0;\mathcal{G}^{\prime})+S(r,\mathcal{F})+S(r,\mathcal{G}),\eeas where $\ol N_{0}(r,0;F^{\prime})$ is the reduced counting function of those zeros of $F^{\prime}$ which are not the zeros of $F(F-1)$ and $\ol N_{0}(r,0;G^{\prime})$ is similarly defined.\end{lem}
\begin{lem}\label{lem3.8}\cite{Lahiri & Complex Var & 2001}
	If two non-constant meromorphic functions $\mathcal{F}$, $\mathcal{G}$ share $(1,2)$, then \beas \ol N_0(r,0;\mathcal{G^{\prime}})+\ol N(r,1;\mathcal{G}\mid \geq2)+\ol N_*(r,1;\mathcal{F},\mathcal{G})\leq\ol N(r,\infty;\mathcal{G})+\ol N(r,0;\mathcal{G})+S(r,\mathcal{G}),\eeas where $\ol N_0(r,0;\mathcal{G^{\prime}})$ is the reduced counting function of those zeros of $\mathcal{G^{\prime}}$ which are not the zeros of $\mathcal{G}(\mathcal{G}-1)$.
\end{lem}
\begin{lem}\label{lem3.9}\cite{Yang & 1993}
	Let f and g be two non-constant meromorphic functions. Then
	\beas N\left(r,\infty;\frac{f}{g}\right)-N\left(r,\infty;\frac{g}{f}\right)= N(r,\infty;f)+N(r,0;g)-N(r,\infty;g)-N(r,0;f).\eeas
\end{lem}
\begin{lem}\label{lem3.10}
	Let $f(z)$ be a transcendental entire function of zero-order, $q\in\mathbb{C}-\{0\}$ and $n, s\in\mathbb{N}$. If $\Phi(z)=f^nP(f)L(z,f)^s$, then $$(n+m)T(r,f)\leq T(r,\Phi)-N(r,0;L(z,f)^s )+S(r,f).$$
\end{lem}
\begin{proof}
	Using the first fundamental theorem of Nevalinna and Lemmas \ref{lem3.1} and \ref{lem3.9}, we have 
	\beas && (n+m+s)T(r,f)= m(r,f^{n+s}P(f))=m\left(r,\frac{\Phi(z)f(z)^{s}}{L(z,f)^{s}}\right)\\&\leq& m(r,\Phi(z))+m\left(r,\frac{f(z)^{s}}{L(z,f)^{s}}\right)+S(r,f)\\&\leq& m(r,\Phi(z))+T\left(r,\frac{f(z)^{s}}{L(z,f)^{s}}\right)-N\left(r,\infty;\frac{f(z)^{s}}{L(z,f)^{s}}\right)+S(r,f)\\&\leq& m(r,\Phi(z))+T\left(r,\frac{L(z,f)^{s}}{f(z)^{s}}\right)-N\left(r,\infty;\frac{f(z)^{s}}{L(z,f)^{s}}\right)+S(r,f)\\&\leq& m(r,\Phi(z))+ N(r,0;f(z)^{s})+N(r, 0;L(z,f)^{s})+S(r,f)\\&\leq& m(r,\Phi(z))+ sT(r,f)+N(r, 0;L(z,f)^{s})+S(r,f).\eeas This implies that \beas (n+m)T(r,f)\leq T(r,\Phi)-N(r,0;L(z,f)^s )+S(r,f).\eeas
\end{proof}
In a similar way as done in the proof of Lemma \ref{lem3.10}, we can prove the following lemma.
\begin{lem}\label{lem3.10a}
	Let $f(z)$ be a transcendental entire function of zero-order, $q\in\mathbb{C}-\{0\}$ and $n, s\in\mathbb{N}$. If $\Phi(z)=P(f)L(z,f)^s$, then $$mT(r,f)\leq T(r,\Phi)-N(r,0;L(z,f)^s )+S(r,f).$$
\end{lem}
\begin{lem}\label{lem3.11}\cite{Zhang & 2005}
	Let $f$ be a non-constant meromorphic function and $p,k\in\mathbb{N}.$ Then
	\beas N_p(r,0;f^{(k)}\leq T(r,f^{(k)})-T(r,f)+N_{p+k}(r, 0;f)+S(r,f),\eeas
		\beas N_p(r,0;f^{(k)}\leq k\ol N(r,\infty;f)+N_{p+k}(r, 0;f)+S(r,f).\eeas
\end{lem}
\begin{lem}\label{lem3.12}\cite{Alzahare & Yi & 2004}
	If $\mathcal{F}$, $\mathcal{G}$ be two non-constant meromorphic functions such that they share $(1,1)$.
	Then \beas 2\ol N_L(r,1;\mathcal{F})+2\ol N_L(r,1;\mathcal{G})+\ol N_E^{(2}(r,1;\mathcal{F})-\ol N_{\mathcal{F}>2}(r,1;\mathcal{G})\leq N(r,1;\mathcal{G})-\ol N(r,1;\mathcal{G}).\eeas
\end{lem}
\begin{lem}\label{lem3.13}\cite{Banerjee & 2005}
	If two non-constant meromorphic functions $\mathcal{F}$, $\mathcal{G}$ share $(1,1)$, then 
	\beas \ol N_{\mathcal{F}>2}(r,1;\mathcal{G})\leq \frac{1}{2}(\ol N(r,0;\mathcal{F})+\ol N(r,\infty;\mathcal{F})-N_{0}(r,0;\mathcal{F}^{\prime}))+S(r,\mathcal{F}),\eeas where $N_{0}(r,0;\mathcal{F}^{\prime}))$ is the counting function of those zeros of $\mathcal{F}^{\prime}$ which are not the zeros of $\mathcal{F}(\mathcal{F}-1)$.
\end{lem}
\begin{lem}\label{lem3.14}\cite{Banerjee & 2005}
	Let $\mathcal{F}$ and $\mathcal{G}$ be two non-constant meromorphic functions sharing $(1,0)$. Then \beas &&\ol N_L(r,1;\mathcal{F} )+2\ol N_L(r,1;\mathcal{G})+\ol N_E^{(2}(r,1;\mathcal{F})-\ol N_{\mathcal{F}>1}(r,1;\mathcal{G})-\ol N_{\mathcal{G}>1}(r,1;\mathcal{F})\\&&\leq N(r,1;\mathcal{G})-\ol N(r,1;\mathcal{G}).\eeas
\end{lem}
\begin{lem}\label{lem3.15}\cite{Banerjee & 2005}
		If $\mathcal{F}$ and $\mathcal{G}$ share $(1,0)$, then \beas \ol N_{L}(r,1;\mathcal{F})\leq \ol N(r,0;\mathcal{F})+\ol N(r,\infty;\mathcal{F})+S(r,\mathcal{F})\eeas
		\beas \ol N_{\mathcal{F}>1}(r,1;\mathcal{G})\leq \ol N(r,0;\mathcal{F} )+\ol N(r,\infty;\mathcal{F})-N_0(r,0;\mathcal{F}^{\prime})+S(r,\mathcal{F}).\eeas Similar inequalities holds for $\mathcal{G}$ also.
\end{lem}

\section{\textbf{Proofs of the theorems}}
\begin{proof}[Proof of Theorem \ref{t1}]
	Suppose $\mathcal{F}=\mathcal{F}_1^{(k)}$, where $\mathcal{F}_1=P(f)L(z,f)^{s}$.
	Let us first suppose that $f$ is a transcendental entire function of zero order. On the contrary, we assume that $\mathcal{F}-\alpha(z)$ has only finitely many zeros. In view of Lemmas \ref{lem3.1}, \ref{lem3.10a}, \ref{lem3.11} and by the Second Theorem for small functions (see \cite{Yamanoi & 2002}), we get \beas mT(r,f)&\leq& T(r,P(f)L(z,f)^{s})-N(r,0;L(z,f)^s )+S(r,f)\\&\leq& T(r,\mathcal{F})+N_{k+1}(r,0;P(f)L(z,f)^s )-\ol N(r,0;\mathcal{F})\\&&-N(r,0;L(z,f)^{s})+S(r,f)\\&\leq& \ol N(r,0;\mathcal{F})+\ol N(r,\alpha;\mathcal{F})+\ol N(r,\infty;\mathcal{F})+N(r,0;L(z,f)^s )\\&& N_{k+1}(r,0;P(f))-\ol N(r,0;\mathcal{F})-N(r,0;L(z,f)^{s})+S(r,f)\\&\leq& N_{k+1}(r,0;P(f))+S(r,f)\\&\leq& (\Gamma_{1}+km_2)T(r,f)+S(r,f),\eeas which is not possible since $n\geq \Gamma_{1}+km_2$.\par 
	Suppose $f$ is a transcendental meromorphic function function of zero order. Now, \beas (m+s)T(r,f)&=&T(r,f^{s}P(f))=T\left(r,\frac{\mathcal{F}_1f^s}{L(z,f)^s}\right)+S(r,f)\\&\leq& T(r,\mathcal{F}_1)+T\left(r,\frac{f^s}{L(z,f)^s}\right)+S(r,f)\\&\leq& T(r,\mathcal{F}_1)+T\left(r,\frac{L(z,f)^s}{f^s}\right)+S(r,f)\\&\leq& T(r,\mathcal{F}_1)+sN\left(r,\infty;\frac{L(z,f)}{f}\right)+sm\left(r,\frac{L(z,f)}{f}\right)+S(r,f)\\&\leq& T(r,\mathcal{F}_1)+2sT(r,f)+S(r,f).\eeas i.e., \beas && (m-s)T(r,f)\leq T(r,\mathcal{F}_1)+S(r,f)\\ &\leq& T(r,\mathcal{F})+N_{k+1}(r,0;P(f)L(z,f)^s)-\ol N(r,0;\mathcal{F})+S(r,f)\\&\leq& \ol N(r,0;\mathcal{F})+\ol N(r,\alpha;\mathcal{F})+\ol N(r,\infty;\mathcal{F})+N_{k+1}(r,0;P(f)L(z,f)^s)\\&&-\ol N(r,0;\mathcal{F})+S(r,f)\\&\leq& \ol N(r,\infty;P(f)L(z,f)^s)+N_{k+1}(r,0;P(f)L(z,f)^s)+S(r,f)\\&\leq& (2+\chi_{_{b_{0}}})\ol N(r,\infty;f)+(\Gamma_{1}+km_2)\ol N(r,0;f)+(1+\chi_{_{b_{0}}})sT(r,f)+S(r,f)\\&\leq&(\Gamma_{1}+km_2+2+\chi_{_{b_{0}}}+s(1+\chi_{_{b_{0}}}))T(r,f)+S(r,f).\eeas i.e., \beas mT(r,f)\leq (\Gamma_{1}+km_2+2s+(1+s)\chi_{_{b_{0}}}+2)T(r,f)+S(r,f),\eeas which is not possible since $m>\Gamma_{1}+km_2+2s+(1+s)\chi_{_{b_{0}}}+2$ and hence the theorem is proved.
\end{proof}
\begin{proof}[Proof of Theorem \ref{t2}]
Let $\mathcal{F}=\displaystyle\frac{f^nP(f)L(z,f)}{p(z)}$ and $\mathcal{G}=\displaystyle\frac{g^nP(g)L(z,g)}{p(z)}$. From the given condition it follows that $\mathcal{F}$, $\mathcal{G}$ share $(1,2)$ except for the zeros of $p(z)$.\par 
\textbf{Case 1:} Let $\mathcal{H}\not\equiv 0$. From \ref{e3.1}, we obtain \bea \label{e4.1}N(r,\infty;\mathcal{H})&\leq& \nonumber\ol N(r,0;\mathcal{F}\mid\geq 2)+\ol N(r,0;\mathcal{G}\mid\geq 2)+\ol N_*(r,1;\mathcal{F},\mathcal{G})\\&&+\ol N_0(r,0;\mathcal{F}^{\prime})+\ol N_0(r,0;\mathcal{G}^{\prime}).\eea If $z_0$ be a simple zero of $\mathcal{F}-1$ such that $p(z_0)\neq0$, then $z_0$ is also a simple zero of $\mathcal{G}-1$ and hence a zero of $\mathcal{H}$. So \bea\label{e4.2} N(r,1\mathcal{F}\mid=1)\leq N(r,0;\mathcal{H})\leq N(r,\infty;\mathcal{H})+S(r,f)+S(r,g).\eea
Using (\ref{e4.1}) and (\ref{e4.2}), we get \bea \label{e4.3}\ol N(r,1;\mathcal{F})&=&N(r,1;\mathcal{F}\mid=1)+\ol N(r,1;\mathcal{F}\mid\geq2)\nonumber\\&\leq& N(r,\infty;\mathcal{H})+\ol N(r,1;\mathcal{F}\mid\geq 2)+S(r,f)+S(r,g)\nonumber\\&\leq& \ol N(r,0;\mathcal{F}\mid\geq 2)+\ol N(r,0;\mathcal{G}\mid\geq 2)+\ol N_*(r,1;\mathcal{F},\mathcal{G})+\ol N_0(r,0;\mathcal{F}^{\prime})\nonumber\\&&+\ol N_0(r,0;\mathcal{G}^{\prime})+\ol N(r,1;\mathcal{F}\mid\geq2)+S(r,f)+S(r,g).\eea
 Now, by Lemma \ref{lem3.3}, we obtain 
 \bea \label{e4.4}&& \ol N_{0}(r,0;\mathcal{G}^{\prime})+\ol N(r,1;\mathcal{F}\mid\geq2)+\ol N_*(r,1;\mathcal{F},\mathcal{G})\nonumber\\&\leq& N(r,0;\mathcal{G}^{\prime}\mid\mathcal{G}\neq0)\leq \ol N(r,0;\mathcal{G})+S(r,g).\eea
  Since $g(z)$ and $g(qz+c)$ share $0$ CM, we must have $N\left(r,\infty;\displaystyle\frac{L(z,g)}{g}\right)=0.$ Hence using (\ref{e4.3}), (\ref{e4.4}) and Lemmas \ref{lem3.10}, \ref{lem3.11}, we get from the Second Fundamental Theorem of Nevalinna, we have \beas && (n+m)T(r,f)\leq T(r,\mathcal{F})-N(r,0;L(z,f))+S(r,f)\\&\leq& \ol N(r,0;\mathcal{F})+\ol N(r,1;\mathcal{F})-\ol N_0(r,0;\mathcal{F}^{\prime})-N(r,0;L(z,f))+S(r,f)\\&\leq& N_2(r,0;\mathcal{F})+N_2(r,0;\mathcal{G})-N(r,0;L(z,f))+S(r,f)+S(r,g)\\&\leq& N_2(r,0;f^nP(f)L(z,f))+N_2(r,0;g^nP(g)L(z,g))-N(r,0;L(z,f))\\&&+S(r,f)+S(r,g)\\&\leq& N_2(r,0;f^nP(f))+N(r,0;L(z,f))+N_2(r,0;P(g))+N_2\left(r,0;g^{n+1}\frac{L(z,g)}{g}\right)\\&&-N\left(r,0;L(z,f)\right)+S(r,f)+S(r,g)\\&\leq& 2\ol N(r,0;f)+N(r,0;P(f))+N(r,0;P(g))+2\ol N(r,0;g)+N\left(r,0;\frac{L(z,g)}{g}\right)\\&&+S(r,f)+S(r,g)\\&\leq& (m+2)T(r,f)+(m+2)T(r,g)+T\left(r,\frac{L(z,g)}{g}\right)+S(r,f)+S(r,g)\\&\leq& (m+2)T(r,f)+(m+2)T(r,g)+N\left(r,\infty;\frac{L(z,g)}{g}\right)+m\left(r,\frac{L(z,g)}{g}\right)\\&&+S(r,f)+S(r,g)\\&\leq& (m+2)T(r,f)+(m+2)T(r,g)+S(r,f)+S(r,g).\eeas i.e., \bea\label{e4.5} nT(r,f)\leq 2T(r,f)+(m+2)T(r,g)+S(r,f)+S(r,g).\eea Since $N\left(r,\infty;\displaystyle\frac{L(z,g)}{g}\right)=0$, keeping in view of Lemmas \ref{lem3.1} and \ref{lem3.4}, we get \beas (n+m+1)T(r,g)&=&T(r,g^{n+1}P(g))=m\left(r,\frac{g^{n+1}P(g)}{\mathcal{G}}\right)+m(r,\mathcal{G})\\&\leq& m\left(r,\frac{g}{L(z,g)}\right)+m(r,\mathcal{G})+O(\log r)\\&\leq& T\left(r,\frac{L(z,g)}{g}\right)+T(r,\mathcal{G})+O(\log r)\\&\leq& m\left(r,\frac{L(z,g)}{g}\right)+T(r,\mathcal{G})+O(\log r)\\&\leq& T(r,\mathcal{G})+O(\log r).\eeas
  In a similar manner, we obtain
  \beas && (n+m+1)T(r,g)\leq T(r,\mathcal{G})+S(r,g)\\&\leq& \ol N(r,0;\mathcal{G})+\ol N(r,1;\mathcal{G})-\ol N_0(r,0;\mathcal{G}^{\prime})+S(r,g)\\&\leq& N_2(r,0;\mathcal{G})+N_2(r,1;\mathcal{G})+S(r,f)+S(r,g)\\&\leq& N_2(r,0;f^nP(f)L(z,f))+N_2(r,0;g^nP(g)L(z,g))+S(r,f)+S(r,g)\\&\leq& 2\ol N(r,0;f)+N(r,0;P(f))+N(r,0;L(z,f))+N(r,0;P(g))\\&&+N_2\left(r,0;g^{n+1}\frac{L(z,g)}{g}\right)+S(r,f)+S(r,g)\\&\leq& (m+2)T(r,f)+T(r,L(z,f))+(m+2)T(r,g)+N\left(r,0;\frac{L(z,g)}{g}\right)\\&&+S(r,f)+S(r,g)\\&\leq& (m+2)T(r,f)+m\left(r,\frac{L(z,f)}{f}\right)+m(r,f)+T\left(r,\frac{L(z,g)}{g}\right)\\&&+(m+2)T(r,g)+S(r,f)+S(r,g)\\&\leq&(m+2)T(r,f)+m\left(r,\frac{L(z,f)}{f}\right)+m(r,f)+N\left(r,\infty;\frac{L(z,g)}{g}\right)\\&&m\left(r,\frac{L(z,g)}{g}\right)+(m+2)T(r,g)+S(r,f)+S(r,g)\\&\leq& (m+3)T(r,f)+(m+2)T(r,g)+S(r,f)+S(r,g).\eeas i.e., \bea \label{e4.6} nT(r,g)\leq (m+3)T(r,f)+T(r,g)+S(r,f)+S(r,g).\eea
  Combining (\ref{e4.5}) and (\ref{e4.6}), we obtain
  \beas (n-m-5)T(r,f)+(n-m-3)T(r,g)\leq S(r,f)+S(r,g),\eeas which contradicts to the fact that $n\geq m+5$.\par 
 \noindent\textbf{ Case 2:} Suppose $H\equiv 0$. Then by integration, we get \bea\label{e4.7} \frac{1}{\mathcal{F}-1}=\frac{A}{\mathcal{G}-1}+B,\eea where $A, B$ are constant with $A\neq0$. From (\ref{e4.7}), it can be easily seen that $\mathcal{F}, \mathcal{G}$ share $(1,\infty)$. We now consider the following three sub-cases.\par 
\noindent \textbf{ Subcase 2.1:} Let $B\neq0$ and $A\neq B$. If $B=-1$, then from (\ref{e4.7}), we have $\mathcal{F}=\displaystyle\frac{-A}{\mathcal{G}-A-1}$.
Therefore, $\ol N(r, A+1;\mathcal{G})=\ol N(r,\infty;\mathcal{F})=N(r, 0;p)=S(r,g)$. So, in view of Lemma \ref{lem3.10} and the second fundamental theorem, we get \beas (n+m)T(r,g)&\leq& T(r, g^nP(g)L(z,g))-N(r,0;L(z,g))+S(r,g)\\&\leq& T(r,\mathcal{G})-N(r,0;L(z,g))+S(r,g)\\&\leq& \ol N(r,0;\mathcal{G})+\ol N(r,A+1;\mathcal{G})-N(r,0;L(z,g))+S(r,g)\\&\leq&(m+1)T(r,g)+S(r,g),\eeas which is a contradiction since $n\geq m+5$.\par 
If $B\neq -1$, then from (\ref{e4.7}), we get $\mathcal{F}-\left(1+\displaystyle\frac{1}{B}\right)=-\displaystyle\frac{A}{B^2\left(\mathcal{G}+\displaystyle\frac{A-B}{B}\right)}$. Therefore, $\ol N\left(r,\displaystyle\frac{B-A}{B};\mathcal{G}\right)=N(r,0;p(g))=O(\log r)=S(r,g)$. Using Lemmas \ref{lem3.11}, \ref{lem3.10} and the same argument as used in the case when $B =-1$ we get a contradiction.\par 
\noindent \textbf{ Subcase 2.2:} Let $B\neq0$ and $A=B$. If $B=-1$, then from (\ref{e4.7}), we have $\mathcal{F}\mathcal{G}\equiv1$. i.e., \bea \label{e4.8} f^nP(f)L(z,f)g^nP(g)L(z,g)\equiv p^2(z).\eea Keeping in view of (\ref{e4.8}) and $\text{deg}(p)<(n-1)/2$, we can say that $f$ and $g$ have zeros. Since $f$ and $g$ are of zero orders,$f$ and $g$ both must be constants which contradicts to our assumption. Therefore, (\ref{e4.8}) is not possible.\par If $B\neq-1$, from (\ref{e4.7}), we have $\displaystyle\frac{1}{\mathcal{F}}=\displaystyle\frac{A\mathcal{G}}{(1+A)\mathcal{G}-1}$. Hence, \beas \ol N\left(r,\frac{1}{1+A};\mathcal{G}\right)=\ol N(r,0;\mathcal{F})+S(r,f).\eeas So in view of Lemmas  \ref{lem3.1} and \ref{lem3.10} and the second fundamental theorem, we obtain \beas (n+m)T(r,g)&\leq& T(r,g^nP(g)L(z,g))-N(r,0;L(z,g))+S(r,g)\\&\leq&T(r, \mathcal{G})-N(r,0;L(z,g))+S(r,g)\\&\leq& \ol N(r,0;\mathcal{G})+\ol N\left(r,\frac{1}{1+A};\mathcal{G}\right)-N(r,0;L(z,g))+S(r,g)\\&\leq& \ol N(r,0;\mathcal{G})+\ol N(r,0;P(g))+\ol N(r,0;\mathcal{F})+S(r,g)\\&\leq& (m+1)T(r,g)+(m+2+\chi_{_{b_{0}}})T(r,f)+S(r,g).\eeas Therefore, \beas nT(r,g)\leq (m+3+\chi_{_{b_{0}}})T(r,g)+S(r,g),\eeas which is a contradiction since $n\geq m+5$.\par 
\noindent \textbf{ Subcase 2.3:} Let $B=0$. Then from (\ref{e4.7}), we get \bea\label{e4.9} \mathcal{F}=\frac{\mathcal{G}+A-1}{A}.\eea If $A\neq 1$, we obtain $\ol N(r,1-A;\mathcal{G})=\ol N(r,0;\mathcal{F})$. Therefore, we can similarly get a contradiction as in Subcase 2.2. Hence $A=1$ and from (\ref{e4.9}), we get \beas\mathcal{F}\equiv\mathcal{G}.\eeas i.e., \bea\label{e4.10} f^nP(f)L(z,f)\equiv g^nP(g)L(z,g).\eea
Let $h=f/g$. If $h$ is constant, then (\ref{e4.10}) reduces to \beas h^ng^nP(hg)L(z,hg)\equiv  g^nP(g)L(z,g).\eeas After simple calculation, we obtain \beas [a_mg^{n+m}(h^{n+m+1}-1)+a_{m-1}g^{n+m}(h^{n+m}-1)+\ldots+a_0g^{n}(h^{n+1}-1)]L(z,g)\equiv 0.\eeas Since $g$ is non-constant, we must have $h^d=1$, where $d=\text{CGD} (n+m+1, n+m,\ldots,n+m+1-i,\ldots,n+1)$ and $a_{m-i}\neq0$ for some $i=0,1,\ldots,m$. Hence $f(z)=tg$ for a constant $t$ such that $t^d=1$, where $d$ is mentioned above. If $h$ is not constant, then $f(z)$, $g(z)$ satisfy the algebraic difference equation $\mathcal{A}(w_1,w_2)\equiv 0$, where \beas \mathcal{A}(w_1,w_2)=w_1^nP(w_1)L(z,w_1)-w_2^nP(w_2)L(z,w_2).\eeas 
\end{proof}
\begin{proof}[Proof of Theorem \ref{t3}]
	Let $\mathcal{F}(z)=(P(f)L(z,f))^{(k)}$ and $\mathcal{G}(z)=(P(g)L(z,g))^{(k)}$. It follows that $\mathcal{F}$ and $\mathcal{G}$ share $(1,l)$.\par 
	\noindent\textbf{ Case 1:} Suppose $\mathcal{H}\not\equiv 0$.\par 
\textbf{ $(i)$}. Let $l\geq 2$. Using Lemmas \ref{lem3.5}, \ref{lem3.7} and \ref{lem3.8}, we get
	\bea \label{e4.11}&& \ol N(r,1;\mathcal{F})=N(r,1;\mathcal{F}\mid=1)+\ol N(r,1;\mathcal{F}\mid\geq2)\nonumber\\&\leq& \ol N(r,0;\mathcal{F}\mid\geq2)+\ol N(r,0;\mathcal{G}\mid\geq2)+\ol N_{*}(r,1;\mathcal{F},\mathcal{G})+\ol N(r,1;\mathcal{F}\mid\geq2)\nonumber\\&&\ol N_0(r,0;\mathcal{F}^{\prime})+\ol N_0(r,0;\mathcal{G}^{\prime})+S(r,f)+S(r,g)\nonumber\\&\leq&\ol N(r,0;\mathcal{F}\mid\geq2)+\ol N(r,0;\mathcal{G}\mid\geq2)+\ol N_0(r,0;\mathcal{F}^{\prime})+\ol N(r,0;\mathcal{G})\nonumber\\&&S(r,f)+S(r,g).\eea Hence, using (\ref{e4.11}), Lemmas \ref{lem3.1}, \ref{lem3.10a}, \ref{lem3.11}, we get from the second fundamental theorem
	\bea\label{e4.12} && mT(r,f)\leq T(r,P(f)L(z,f))-N(r,0;L(z,f))+S(r,f)\nonumber\\&\leq& T(r,\mathcal{F})+N_{k+2}(r,0;P(f)L(z,f))-N_2(r,0;\mathcal{F})-N(r,0;L(z,f))+S(r,f)\nonumber\\&\leq& \ol N(r,0;\mathcal{F})+\ol N(r,1;\mathcal{F})+N_{k+2}(r,0;P(f)L(z,f))-N_2(r,0;\mathcal{F})-\ol N_0(r,0;\mathcal{F}^{\prime})\nonumber\\&&-N(r,0;L(z,f))+S(r,f)\nonumber\\&\leq& N_2(r,0;\mathcal{G})+N_{k+2}(r,0;P(f)L(z,f))-N(r,0;L(z,f))+S(r,f)+S(r,g)\nonumber\\&\leq& N_{k+2}(r,0;P(f)L(z,f))+N_{k+2}(r,0;P(g)L(z,g))-N(r,0;L(z,f))\nonumber\\&&+S(r,f)+S(r,g)\nonumber\\&\leq& N_{k+2}(r,0;P(f))+N_{k+2}(r,0;P(g))+N(r,0;L(z,g))+S(r,f)+S(r,g)\nonumber\\&\leq& (m_1+2m_2+km_2)(T(r,f)+T(r,g))+T(r,L(z,g))+S(r,f)+S(r,g)\nonumber\\&\leq& (m_1+2m_2+km_2)(T(r,f)+T(r,g))+m\left(r,\frac{L(z,g)}{g}\right)+m(r,g)\nonumber\\&&S(r,f)+S(r,g)\nonumber\\&\leq& (m_1+2m_2+km_2)(T(r,f)+T(r,g))+T(r,g)+S(r,f)+S(r,g).\eea Similarly,\bea \label{e4.13}mT(r,g)&\leq& (m_1+2m_2+km_2)(T(r,f)+T(r,g))+T(r,f)+S(r,f)\nonumber\\&&+S(r,g).\eea Combining (\ref{e4.12}) and (\ref{e4.13}), we have \beas m(T(r,f)+T(r,g))\leq (2\Gamma_{0}+2km_2+1)(T(r,f)+T(r,g))+S(r,f)+S(r,g),\eeas which is a contradiction since $m>2\Gamma_{0}+2km_2+1$.\par 
	\textbf{$(ii)$.} Let $l=1$, using Lemmas \ref{lem3.3}, \ref{lem3.5}, \ref{lem3.7}, \ref{lem3.12}, \ref{lem3.13}, we get \bea \label{e4.14} \ol N(r,1;\mathcal{F})&\leq& N(r,1;\mathcal{F}\mid=1)+\ol N_{L}(r,1;\mathcal{F})+\ol N_{L}(r,1;\mathcal{G})+\ol N_{E}^{(2}(r,1;\mathcal{F})\nonumber\\&\leq& \ol N(r,0;\mathcal{F}\mid\geq2)+\ol N(r,0;\mathcal{G}\mid\geq2)+\ol N_*(r,1;\mathcal{F},\mathcal{G})+\ol N_{L}(r,1;\mathcal{F})\nonumber\\&&+\ol N_L(r,1;\mathcal{G})+\ol N_E^{(2}(r,1;\mathcal{F})+\ol N_0(r,0;\mathcal{F}^{\prime})+\ol N_0(r,0;\mathcal{G}^{\prime})\nonumber\\&&+S(r,f)+S(r,g)\nonumber\\&\leq& \ol N(r,0;\mathcal{F}\mid\geq2)+\ol N(r,0;\mathcal{G}\mid\geq2)+2\ol N_{L}(r,1;\mathcal{F})\nonumber+2\ol N_L(r,1;\mathcal{G})\nonumber\\&&+\ol N_E^{(2}(r,1;\mathcal{F})+\ol N_0(r,0;\mathcal{F}^{\prime})+\ol N_0(r,0;\mathcal{G}^{\prime})\nonumber+S(r,f)+S(r,g)\nonumber\\&\leq& \ol N(r,0;\mathcal{F}\mid\geq2)+\ol N(r,0;\mathcal{G}\mid\geq2)+\ol N_{\mathcal{F}>2}(r,1;\mathcal{G})+N(r,1;\mathcal{G})\nonumber\\&&-\ol N(r,1;\mathcal{G})+\ol N_0(r,0;\mathcal{F}^{\prime})+\ol N_0(r,0;\mathcal{G}^{\prime})\nonumber+S(r,f)+S(r,g)\nonumber\\&\leq& \ol N(r,0;\mathcal{F}\mid\geq2)+\ol N(r,0;\mathcal{G}\mid\geq2)+\frac{1}{2}\ol N(r,0;\mathcal{F})+N(r,1;\mathcal{G})\nonumber\\&&-\ol N(r,1;\mathcal{G})+\ol N_0(r,0;\mathcal{F}^{\prime})+\ol N_0(r,0;\mathcal{G}^{\prime})\nonumber+S(r,f)+S(r,g)\nonumber\\&\leq& \ol N(r,0;\mathcal{F}\mid\geq2)+\ol N(r,0;\mathcal{G}\mid\geq2)+\frac{1}{2}\ol N(r,0;\mathcal{F})+\ol N_0(r,0;\mathcal{F}^{\prime})\nonumber\\&& N(r,0;\mathcal{G}^{\prime}\mid\mathcal{G}\neq0)+S(r,f)+S(r,g)\nonumber\\&\leq& \ol N(r,0;\mathcal{F}\mid\geq2)+\frac{1}{2}\ol N(r,0;\mathcal{F})+N_2(r,0;\mathcal{G})+\ol N_0(r,0;\mathcal{F}^{\prime})\nonumber\\&&+S(r,f)+S(r,g).\eea Hence using (\ref{e4.14}), Lemmas \ref{lem3.1}, \ref{lem3.10a}, \ref{lem3.11} and the second fundamental theorem, we get 
	\bea \label{e4.15} mT(r,f)&\leq& T(r, P(f)L(z,f))-N(r,0;L(z,f))+S(r,f)\nonumber\\&\leq& T(r,\mathcal{F})+N_{k+2}(r,0;P(f)L(z,f))-N_2(r,0;\mathcal{F})-\ol N_0(r,0;\mathcal{F}^{\prime})\nonumber\\&&-N(r,0;L(z,f))+S(r,f)\nonumber\\&\leq& \ol N(r,0;\mathcal{F})+\ol N(r,1;\mathcal{F})+N_{k+2}(r,0;P(f)L(z,f))-N_2(r,0;\mathcal{F})\nonumber\\&&-\ol N_0(r,0;\mathcal{F}^{\prime})-N(r,0;L(z,f))+S(r,f)+S(r,g)\nonumber\\&\leq& N_{k+2}(r,0;P(f)L(z,f))+\frac{1}{2}\ol N(r,0;\mathcal{F})+N_2(r,0;\mathcal{G})-N(r,0;L(z,f))\nonumber\\&&+S(r,f)+S(r,g)\nonumber\\&\leq& N_{k+2}(r,0;P(f)L(z,f))+N_{k+2}(r,0;P(g)L(z,g))-N(r,0;L(z,f))\nonumber\\&&+\frac{1}{2}N_{k+1}(r,0;P(f)L(z,f))++S(r,f)+S(r,g)\nonumber\\&\leq& (m_1+2m_2+km_2)(T(r,f)+T(r,g))+\frac{1}{2}(m_1+m_2+km_2)T(r,f)\nonumber\\&&+T(r,g)+\frac{1}{2}T(r,f)+S(r,f)+S(r,g).\eea In a similar manner, we get 
	\bea \label{e4.16} mT(r,g)&\leq& (m_1+2m_2+km_2)(T(r,f)+T(r,g))+\frac{1}{2}(m_1+m_2+km_2)T(r,g)\nonumber\\&&+T(r,f)+\frac{1}{2}T(r,g)+S(r,f)+S(r,g).\eea Combining (\ref{e4.15}) and (\ref{e4.16}), we get \beas && m(T(r,f)+T(r,g))\\&\leq& \left(2(m_1+2m_2+km_2)+\frac{1}{2}(m_1+m_2+km_2)+\frac{3}{2}\right)(T(r,f)+T(r,g))\nonumber\\&& S(r,f)+S(r,g).\eeas i.e., \beas && m(T(r,f)+T(r,g))\\&\leq& \frac{1}{2}(4\Gamma_{0}+\Gamma_{1}+5km_2+3)(T(r,f)+T(r,g))+S(r,f)+S(r,g), \eeas which is a contradiction since $m>\displaystyle\frac{1}{2}(4\Gamma_{0}+\Gamma_{1}+5km_2+3)$.\par 
	\noindent\textbf{ $(iii)$.} Let $l=0$. Using Lemmas \ref{lem3.3}, \ref{lem3.5}, \ref{lem3.7}, \ref{lem3.14}, \ref{lem3.15}, we get \bea \label{e4.17} \ol N(r,1;\mathcal{F})&\leq& N_{E}^{1)}(r,1;\mathcal{F})+\ol N_L(r,1;\mathcal{F})+\ol N_L(r,1;\mathcal{G})+\ol N_{E}^{(2}(r,1;\mathcal{F})\nonumber\\&\leq& \ol N(r,0;\mathcal{F}\mid\geq 2)+\ol N(r,0;\mathcal{G}\mid\geq 2)+\ol N_*(r,1;\mathcal{F},\mathcal{G})+\ol N_L(r,1;\mathcal{F})\nonumber\\&&+\ol N_L(r,1;\mathcal{G})+\ol N_{E}^{(2}(r,1;\mathcal{F})+\ol N_0(r,0;\mathcal{F}^{\prime})+\ol N_0(r,0;\mathcal{G}^{\prime})\nonumber\\&\leq& \ol N(r,0;\mathcal{F}\mid\geq 2)+\ol N(r,0;\mathcal{G}\mid\geq 2)+2\ol N_L(r,1;\mathcal{F})+2\ol N_L(r,1;\mathcal{G})\nonumber\\&&+\ol N_{E}^{(2}(r,1;\mathcal{F})+\ol N_0(r,0;\mathcal{F}^{\prime})+\ol N_0(r,0;\mathcal{G}^{\prime})+S(r,f)+S(r,g)\nonumber\\&\leq& \ol N(r,0;\mathcal{F}\mid\geq 2)+\ol N(r,0;\mathcal{G}\mid\geq 2)+\ol N_L(r,1;\mathcal{F})+\ol N_{\mathcal{F}>1}(r,1;\mathcal{G})\nonumber\\&&+\ol N_{\mathcal{G}>1}(r,1;\mathcal{F})+N(r,1;\mathcal{G})-\ol N(r,1;\mathcal{G})+\ol N_0(r,0;\mathcal{F}^{\prime})\nonumber\\&&+\ol N_0(r,0;\mathcal{G}^{\prime})+S(r,f)+S(r,g)\nonumber\\&\leq& \ol N(r,0;\mathcal{F}\mid\geq 2)+\ol N(r,0;\mathcal{G}\mid\geq 2)+2\ol N(r,0;\mathcal{F})+\ol N(r,0;\mathcal{G})\nonumber\\&&+N(r,1;\mathcal{G})-\ol N(r,1;\mathcal{G})+\ol N_0(r,0;\mathcal{F}^{\prime})+\ol N_0(r,0;\mathcal{G}^{\prime})\nonumber\\&\leq& N_2(r,0;\mathcal{F})+\ol N(r,0;\mathcal{F})+N_2(r,0;\mathcal{G})+N(r,1;\mathcal{G})-\ol N(r,1;\mathcal{G})\nonumber\\&&+ N_0(r,0;\mathcal{G}^{\prime})+N_0(r,0;\mathcal{F}^{\prime})+S(r,f)+S(r,g)\nonumber\\&\leq& N_2(r,0;\mathcal{F})+\ol N(r,0;\mathcal{F})+N_2(r,0;\mathcal{G})+N(r,0;\mathcal{G}^{\prime}\mid \mathcal{G}\neq0)\nonumber\\&&+N_0(r,0;\mathcal{F}^{\prime})+S(r,f)+S(r,g)\nonumber\\&\leq& N_2(r,0;\mathcal{F})+\ol N(r,0;\mathcal{F})+N_2(r,0;\mathcal{G})+\ol N(r,0;\mathcal{G})\nonumber\\&&+N_0(r,0;\mathcal{F}^{\prime})+S(r,f)+S(r,g).\eea Hence using (\ref{e4.17}), Lemmas \ref{lem3.1}, \ref{lem3.10a}, \ref{lem3.11} and the second fundamental theorem, we get \bea \label{e4.18} mT(r,f)&\leq& T(r,P(f)L(z,f))-N(r,0;L(z,f))+S(r,f)\nonumber\\&\leq& T(r,\mathcal{F})+N_{k+2}(r,0;P(f)L(z,f))-N_2(r,0;\mathcal{F})-N(r,0;L(z,f))\nonumber\\&&+S(r,f)\nonumber\\&\leq& \ol N(r,0;\mathcal{F})+\ol N(r,1;\mathcal{F})+N_{k+2}(r,0;P(f)L(z,f))-N_2(r,0;\mathcal{F})\nonumber\\&&-N(r,0;L(z,f))+S(r,f)\nonumber\\&\leq& N_{k+2}(r,0;P(f)L(z,f))+2\ol N(r,0;\mathcal{F})+N_2(r,0;\mathcal{G})+\ol N(r,0;\mathcal{G})\nonumber\\&&-N(r,0;L(z,f))+S(r,f)+S(r,g)\nonumber\\&\leq& N_{k+2}(r,0;P(f)L(z,f))+2N_{k+1}(r,0;P(f)L(z,f))\nonumber\\&&+N_{k+2}(r,0;P(g)L(z,g))+N_{k+1}(r,0;P(g)L(z,g))-N(r,0;L(z,f))\nonumber\\&&+S(r,f)+S(r,g)\nonumber\\&\leq& N_{k+2}(r,0;P(f))+2N_{k+1}(r,0;P(f)L(z,f))+N_{k+2}(r,0;P(g)L(z,g))\nonumber\\&&+N_{k+1}(r,0;P(g)L(z,g)+S(r,f)+S(r,g)\nonumber\\&\leq& (m_1+2m_2+km_2+2)(T(r,f)+T(r,g))+2(m_1+m_2+km_2)T(r,f)\nonumber\\&&+(m_1+m_2+km_2)T(r,g)+S(r,f)+S(r,g). \eea In a similar manner we get \bea \label{e4.19} mT(r,g)&\leq& (m_1+2m_2+km_2+2)(T(r,f)+T(r,g))+2(m_1+m_2+km_2)T(r,g)\nonumber\\&&+(m_1+m_2+km_2)T(r,f)+S(r,f)+S(r,g).\eea Combining (\ref{e4.18}) and (\ref{e4.19}), we have \beas && m(T(r,f)+T(r,g))\nonumber\\&\leq& (m_1+2m_2+km_2+2+2(m_1+m_2+km_2)+(m_1+m_2+km_2))(T(r,f)\\&&+T(r,g))+S(r,f)+S(r,g) \\&=& (2\Gamma_{0}+3\Gamma_{1}+5km_2+4)(T(r,f)+T(r,g))+S(r,f)+S(r,g),\eeas which is a contradiction since $m>2\Gamma_{0}+3\Gamma_{1}+5km_2+4$.\par 
	\noindent\textbf{ Case 2:} Let $\mathcal{H}\equiv 0$. By integration, we get \bea\label{e4.20} \frac{1}{\mathcal{F}-1}=\frac{A}{\mathcal{G}-1}+B,\eea where $A, B$ are constant with $A\neq0$. From (\ref{e4.20}), it can be easily seen that $\mathcal{F}, \mathcal{G}$ share $(1,\infty)$. We now consider the following sub-cases.\par 
	\noindent \textbf{ Subcase 2.1:} Let $B\neq0$ and $A\neq B$. If $B=-1$, then from (\ref{e4.20}), we have $\mathcal{F}=\displaystyle\frac{-A}{\mathcal{G}-A-1}$.
	Therefore, $\ol N(r, A+1;\mathcal{G})=\ol N(r,\infty;\mathcal{F})=S(r,f)$. Therefore, using Lemma \ref{lem3.10a} and the second fundamental theorem of Nevalinna, we get \beas && mT(r,g)\leq T(r,P(g)L(z,g))-N(r,0;L(z,g))+S(r,g)\\&\leq& T(r,\mathcal{G})+N_{k+2}(r,0;P(g)L(z,g))-N_2(r,0;\mathcal{G})-N(r,0;L(z,g))+S(r,g)\\&\leq& \ol N(r,0;\mathcal{G})+\ol N(r,A+1;\mathcal{G})+N_{k+2}(r,0;P(g)L(z,g))-N_2(r,0;\mathcal{G})\\&&-N(r,0;L(z,g))+S(r,g)\\&\leq& N_{k+1}(r,0;P(g)L(z,g))+N_{k+2}(r,0;P(g))+N(r,0;L(z,g))-N(r,0;L(z,g))\\&&+S(r,g)\\&\leq& (2m_1+3m_2+2km_2+1)T(r,g)+S(r,g)\\&=&(2\Gamma_{0}+2km_2+1-m_2)T(r,g)+S(r,g),\eeas which is a contradiction since $m>2\Gamma_{0}+2km_2+1$. If $B\neq -1$, then from (\ref{e4.20}), we have $\mathcal{F}=\displaystyle\frac{(B+1)\mathcal{G}-(B-A+1)}{B\mathcal{G}+(A-B)}$ and therefore, $\ol N\left(r,\displaystyle\frac{A-B}{B};\mathcal{G}\right)=\ol N(r,\infty;\mathcal{F})=S(r,f)$. Therefore, in a similar manner as done in the case $B=-1$, we arrive at a contradiction. \par 
	\noindent \textbf{ Subcase 2.2:} Let $B\neq 0$ and $A=B$. If $B\neq-1$, then from (\ref{e4.20}), we have $\displaystyle\frac{1}{\mathcal{F}}=\displaystyle\frac{B\mathcal{G}}{(B+1)\mathcal{G}-1}$, and therefore, $\ol N(r,0;\mathcal{G})=\ol N(r,\infty;\mathcal{F})=S(r,f)$ and $\ol N\left(r,\displaystyle\frac{1}{B+1};\mathcal{G}\right)=\ol N(r,0;\mathcal{F})$. Therefore, using Lemma \ref{lem3.10a} and the second fundamental theorem of Nevalinna, we get \beas && mT(r,g)\leq T(r,P(g)L(z,g))-N(r,0;L(z,g))+S(r,g)\\&\leq& T(r,\mathcal{G})+N_{k+2}(r,0;P(g)L(z,g))-N_2(r,0;\mathcal{G})-N(r,0;L(z,g))+S(r,g)\\&\leq& \ol N(r,0;\mathcal{G})+\ol N\left(r,\frac{1}{B+1};\mathcal{G}\right)+N_{k+2}(r,0;P(g)L(z,g))-N_2(r,0;\mathcal{G})\\&&-N(r,0;L(z,g))+S(r,g)\\&\leq& \ol N(r,0;\mathcal{F})+N_{k+2}(r,0;P(g)L(z,g))-N_2(r,0;\mathcal{G})-N(r,0;L(z,g))+S(r,g)\\&\leq& N_{k+1}(r,0;P(f)L(z,f))+N_{k+2}(r,0;P(g))+N(r,0;L(z,g))-N(r,0;L(z,g))\\&&+S(r,g)\\&\leq& (m_1+2m_2+km_2)T(r,g)+(m_1+m_2+km_2+1)T(r,f)+S(r,f)+S(r,g).\eeas
	Similarly, \beas && mT(r,f)\\&\leq& (m_1+2m_2+km_2)T(r,f)+(m_1+m_2+km_2+1)T(r,g)+S(r,f)+S(r,g).\eeas Combining the above two inequalities, we obtain \beas && m(T(r,f)+T(r,g))\\&\leq& (2\Gamma_{0}+2km_2+1-m_2)(T(r,f)+T(r,g))+S(r,f)+S(r,g),\eeas which is a contradiction since $m>2\Gamma_{0}+2km_2+1$.\par If $B=-1$, then (\ref{e4.20}) reduces to $\mathcal{F}\mathcal{G}\equiv1$. This implies \bea \label{e4.21} (P(f)L(z,f))^{(k)}(P(g)L(z,g))^{(k)}\equiv 1.\eea Suppose that $P(z)=0$ has $t$ roots $\alpha_1, \alpha_2,\ldots,\alpha_t$ with multiplicities $u_1,u_2,\ldots,u_t$, respectively. Then we must have $u_1+u_2+\ldots+u_t=m$. Therefore (\ref{e4.21}) can be rewritten as \bea\label{e4.22} && (a_m(f-\alpha_1)^{u_1})\ldots(f-\alpha_t)^{u_t}L(z,f))^{(k)}(a_m(g-\alpha_1)^{u_1})\ldots(g-\alpha_t)^{u_t}L(z,g))^{(k)}\nonumber\\ &&\equiv  1.\eea Since $f$ and $g$ are entire functions, from (\ref{e4.22}), we can say that $\alpha_1,\alpha_2,\ldots,\alpha_t$ are Picard exceptional values of $f$ and $g$. Since by Picards theorem, an entire function can have atmost one finite exceptional value, all $\alpha_j$'s are equal for $1\leq j\leq t$. Let $P(z)=a_m(z-\alpha)^m$. therefore, (\ref{e4.22}) reduces to \bea\label{e4.23} (a_m(f-\alpha)^mL(z,f))^{(k)}(a_m(g-\alpha)^mL(z,g))^{(k)}\equiv1.\eea Equation (\ref{e4.23}) shows that $\alpha$ is an exceptional value of $f$ and $g$. Since $f$ is an entire function of zero order having an exceptional value $\alpha$, $f$ must be constant, which is not possible since $f$ is assumed to be transcendental and therefore non constant.\par
	\noindent \textbf{ Subcase 2.3:} Let $B=0$. Then (\ref{e4.20}) reduces to $\mathcal{F}=\displaystyle\frac{\mathcal{G}+A-1}{A}$. If $A\neq 1$, then $\ol N(r,1-A;\mathcal{G})=\ol N(r,0;\mathcal{F})$. Proceeding ina similar manner as done in\textbf{ subcase 2.2}, we get a contradiction. Hence $A=1$. Therefore, $\mathcal{F}\equiv \mathcal{G}$. This implies that \bea \label{e4.24} (P(f)L(z,f))^{(k)}\equiv (P(g)L(z,g))^{(k)}.\eea Integrating (\ref{e4.24}) $k$ times, we get \bea \label{e4.25} P(f)L(z,f)=P(g)L(z.g)+p_1(z),\eea where $p_1(z)$ is a polynomial in $z$ of degree atmost $k-1$. Suppose $p_1(z)\not\equiv 0$. Then (\ref{e4.25}) can be written as \bea \label{e4.26} \frac{P(f)L(z,f)}{p_1(z)}=\frac{P(g)L(z,g)}{p_1(z)}+1.\eea Now in view of Lemmas \ref{lem3.1}, \ref{lem3.10a} and the second fundamental theorem, we have 
	\beas mT(r,f)&\leq& T(r,P(f)L(z,f))-N(r,0;L(z,f))+S(r,f)\\&\leq& T\left(r,\frac{P(f)L(z,f)}{p_1(z)}\right)-N(r,0;L(z,f))+S(r,f)\\&\leq& \ol N\left(r,0;\frac{P(f)L(z,f)}{p_1(z)}\right)+\ol N\left(r,\infty;\frac{P(f)L(z,f)}{p_1(z)}\right)\\&&+\ol N\left(r,1;\frac{P(f)L(z,f)}{p_1(z)}\right)-N(r,0;L(z,f))+S(r,f)\\&\leq& \ol N\left(r,0;\frac{P(f)L(z,f)}{p_1(z)}\right)+\ol N\left(r,\infty;\frac{P(f)L(z,f)}{p_1(z)}\right)\\&&+\ol N\left(r,0;\frac{P(g)L(z,g)}{p_1(z)}\right)-N(r,0;L(z,f))+S(r,f)\\&\leq& \ol N(r,0;P(f))+\ol N(r,0;P(g))+\ol N(r,0;L(z,g))+S(r,f)+S(r,g)\\&\leq& (m_1+m_2)(T(r,f)+T(r,g))+T(r,L(z,g))+S(r,f)+S(r,g)\\&\leq& (m_1+m_2)(T(r,f)+T(r,g))+m\left(r,\frac{L(z,g)}{g}\right)+m(r,g)+S(r,f)\\&&+S(r,g)\\&\leq& (m_1+m_2)(T(r,f)+T(r,g))+T(r,g)+S(r,f)+S(r,g).\eeas Similarly we have \beas mT(r,g)\leq (m_1+m_2)(T(r,f)+T(r,g))+T(r,f)+S(r,f)+S(r,g).\eeas
	Combining the last two inequalities, we obtain
	\beas m(T(r,f)+T(r,g))\leq (2m_1+2m_2+1)(T(r,f)+T(r,g))+S(r,f)+S(r,g),\eeas which is a contradiction since $m>2\Gamma_{0}+2km_2+1$. The same arguments also hold for the case $m>(\Gamma_{1}+4\Gamma_{0}+5km_2+3)/2$ and $m>3\Gamma_{1}+2\Gamma_{0}+5km_2+4$. Hence $p_1(z)\equiv 0$ and therefore from (\ref{e4.25}), we have \bea\label{e4.27} P(f)L(z,f)=P(g)L(z,g).\eea Set $h=f/g$. If $h$ is non-constant, from (\ref{e4.27}), we can get that $f$ and $g$ satisfy the	algebraic equation $\mathcal{A}(f, g)= 0$, where $\mathcal{A}(w_1,w_2)=P(w_1)L(z,w_1)-P(w_2)L(z,w_2).$ If $h$ is a constant, substituting $f=gh$ into (\ref{e4.27}), we can get \beas [a_mg^m(h^{m+1}-1)+a_{m-1}g^{m-1}(h^{m}-1)+\ldots+a_0(h-1)]L(z,f)=0.\eeas Then in a similar argument as done in Case 2 in the proof of Theorem 11 in \cite{Xu & Liu & Cao & 2015}, we	obtain $f=tg$ for a constant $t$ such that $t^d=1$; $d=\textit{GCD}(\lambda_0; \lambda_1;\ldots; \lambda_m)$.
\end{proof}
\textbf{Acknowledgments}\\
There is no financial support from any agencies for this work.

\end{document}